\documentclass[11pt]{article}
\usepackage[margin=1.1in,a4paper]{geometry}
\usepackage[dvipsnames]{xcolor}
\usepackage[authoryear,longnamesfirst]{natbib}
\usepackage{amsmath,amssymb,mathtools,amsthm}
\usepackage{amsfonts,bm}
\usepackage{url}
\usepackage[colorlinks=true,citecolor=blue,linkcolor=blue,urlcolor=blue]{hyperref}

\newtheorem{theorem}{Theorem}
\newtheorem{corollary}[theorem]{Corollary}
\newtheorem{lemma}[theorem]{Lemma}
\newtheorem{remark}{Remark}

\newcommand{\R}{\mathbb{R}}
\newcommand{\N}{\mathbb{N}}
\newcommand{\E}{\mathbb{E}}
\newcommand{\dd}{\,\mathrm{d}}
\newcommand{\abs}[1]{\left\lvert #1\right\rvert}

\title{Absolute moment inequalities under quadratic-form positivity}
\author{Zhekai Pang\thanks{Universitat Pompeu Fabra, Barcelona, Spain. Email: \texttt{zhekai.pang@upf.edu}.}}
\date{}

\begin{document}

\maketitle

\begin{abstract}
We prove the open question posed by Zhuang and Hu in Remark 3.1. More generally, we consider symmetric joint probability mass functions and joint densities whose associated quadratic form is non-negative. In this class, for every \(r>0\), the inequality \(\E\abs{X+Y}^{r}\ge \E\abs{X-Y}^{r}\) holds for all distributions with finite \(r\)-th absolute moment if and only if \(0<r\le2\).
\end{abstract}


\section{Introduction}

Let \(X\) and \(Y\) be independent and identically distributed real random variables. The characteristic-function method behind many inequalities of this type goes back to \citet{BahrEsseen1965}, who used the integral representation of \(\abs{x}^r\) for \(0<r<2\) in the study of absolute moments of sums. \citet{An2006} and \citet{Ushakov2011} proved that
\[
   \E\abs{X+Y}^{r}\ge \E\abs{X-Y}^{r},\qquad 0<r<2,
\]
whenever \(\E\abs{X}^{r}<\infty\).

\citet{ZhuangHu2026} considered the following two classes of symmetric bivariate distributions. In the discrete case,
\begin{equation}\label{eq:cauchy-discrete}
   \mathbb P(X=a_i,Y=a_j)=\frac{d_i d_j}{c_i+c_j},
   \qquad i,j\in I,
\end{equation}
where \(I=\{1,\ldots,n\}\) for some \(n\in\N^+\), or \(I=\N\), and \(a_i\in\R\), \(c_i,d_i>0\). In the continuous case,
\begin{equation}\label{eq:cauchy-density}
   f(x,y)=\frac{d(x)d(y)}{c(x)+c(y)},
   \qquad (x,y)\in D^2,
\end{equation}
where \(D\subseteq\R\) is Borel and \(c(x),d(x)>0\). In both displays the right-hand side is assumed to define a valid joint distribution. Zhuang and Hu proved the case \(r=1\). They then asked in their Remark 3.1 whether the same inequality remains true when \((X,Y)\) has joint probability mass function \eqref{eq:cauchy-discrete} or joint density function \eqref{eq:cauchy-density}.

We prove this open question and, more generally, extend the inequality from the Cauchy kernel to a larger class defined by a quadratic-form positivity condition.

\section{Model}

Let \(\mathcal B(\R)\) denote the Borel \(\sigma\)-field on \(\R\).

First consider the discrete case. Let \(I=\{1,\ldots,n\}\) for some \(n\in\N^+\), or let \(I=\N\). Let \((a_i)_{i\in I}\) be support points in \(\R\), listed without repetition. Let \((p_{ij})_{i,j\in I}\) be a symmetric array such that \(p_{ij}\ge0\) and
\[
   \sum_{i\in I}\sum_{j\in I}p_{ij}=1.
\]
Define
\begin{equation}\label{eq:discrete-model}
   \mathbb P(X=a_i,Y=a_j)=p_{ij},
   \qquad i,j\in I.
\end{equation}
Assume that \((p_{ij})\) is non-negative in the quadratic-form sense:
\begin{equation}\label{eq:discrete-positive}
   \sum_{i\in F}\sum_{j\in F}u_i u_j p_{ij}\ge0
\end{equation}
for every finite set \(F\subseteq I\) and every real family \((u_i)_{i\in F}\). In the finite case, this is the usual positive semi-definiteness of the matrix \((p_{ij})\).

Next consider the density case. Let \(D\in\mathcal B(\R)\), and let \(f:D^2\to[0,\infty)\) be measurable and symmetric. Assume
\begin{equation}\label{eq:density-normalization}
   \iint_{D^2}f(x,y)\dd x\dd y=1.
\end{equation}
Thus \(f\) is a joint density function on \(D^2\). Assume also that \(f\) is non-negative in the integral quadratic-form sense: for every measurable \(\delta:D\to\R\),
\[
   \iint_{D^2}\abs{\delta(x)\delta(y)}f(x,y)\dd x\dd y<\infty
\]
implies
\begin{equation}\label{eq:density-positive}
   \iint_{D^2}\delta(x)\delta(y)f(x,y)\dd x\dd y\ge0.
\end{equation}
Equivalently, the integral operator induced by \(f\) has a non-negative quadratic form on its natural domain of functions for which the displayed integral is absolutely convergent.
Let \((X,Y)\) have joint density \(f\), that is,
\begin{equation}\label{eq:density-model}
   \mathbb P((X,Y)\in B)=\iint_{B}f(x,y)\dd x\dd y,
   \qquad B\in\mathcal B(D^2).
\end{equation}

\begin{remark}\label{rem:terminology}
The phrase ``positive semi-definite'' is used here for the quadratic form generated by the joint mass function or density function. This is slightly different from the usual pointwise positive semi-definite kernel condition in functional analysis; see, for example, \citet[Chapter 3]{BergChristensenRessel1984} for the standard kernel terminology. In the discrete case, it reduces to the usual positive semi-definiteness of the matrix \((p_{ij})\). In the density case, it is the integral analogue in \eqref{eq:density-positive}.
\end{remark}

\section{Lemmas}

\begin{lemma}\label{lem:chung}
Let \(0<r<2\). Let
\[
   C_r=\left(\int_{-\infty}^{\infty}\frac{1-\cos u}{\abs{u}^{r+1}}\dd u\right)^{-1}
   =\frac{\Gamma(r+1)}{\pi}\sin\frac{\pi r}{2}.
\]
Then, for every \(z\in\R\),
\[
   \abs{z}^{r}=C_r\int_{-\infty}^{\infty}\frac{1-\cos(tz)}{\abs{t}^{r+1}}\dd t.
\]
\end{lemma}

\begin{proof}
See \citet[Section 3]{BahrEsseen1965} or \citet[Section 6.2, Exercise 8]{Chung2001}.
\end{proof}

\begin{lemma}\label{lem:cauchy-positive}
Let \(D\in\mathcal B(\R)\), let \(\mu\) be a measure on \(\mathcal B(D)\), let \(c:D\to(0,\infty)\) be measurable, and let \(\eta:D\to\R\) be measurable. If
\[
   \iint_{D^2}\frac{\abs{\eta(x)\eta(y)}}{c(x)+c(y)}\mu(\dd x)\mu(\dd y)<\infty,
\]
then
\[
   \iint_{D^2}\frac{\eta(x)\eta(y)}{c(x)+c(y)}\mu(\dd x)\mu(\dd y)\ge0.
\]
\end{lemma}

\begin{proof}
See \citet[Lemma 2.1]{ZhuangHu2026}.
\end{proof}

\section{Main result}

\begin{theorem}\label{thm:main}
Let \(r>0\). The inequality
\[
   \E\abs{X+Y}^{r}\ge \E\abs{X-Y}^{r}
\]
holds for every random vector \((X,Y)\) satisfying either the discrete model \eqref{eq:discrete-model}-\eqref{eq:discrete-positive} or the density model \eqref{eq:density-positive}-\eqref{eq:density-model}, with \(\E\abs{X}^{r}<\infty\), if and only if \(0<r\le2\).
\end{theorem}

\begin{proof}
We prove the two directions in three cases.

First assume that \(0<r<2\). Let \((X,Y)\) satisfy either the discrete model or the density model, and assume that \(\E\abs{X}^{r}<\infty\). In both models the joint distribution is symmetric in \((x,y)\). Therefore \(X\) and \(Y\) have the same marginal distribution, and \(\E\abs{Y}^{r}=\E\abs{X}^{r}<\infty\).

For all real \(x,y\),
\[
   \abs{x+y}^{r}\le A_r(\abs{x}^{r}+\abs{y}^{r}),
   \qquad
   \abs{x-y}^{r}\le A_r(\abs{x}^{r}+\abs{y}^{r}),
\]
where
\[
   A_r=\begin{cases}
      1, & 0<r\le1,\\
      2^{r-1}, & 1<r<2.
   \end{cases}
\]
This is because the function \(t\mapsto t^r\) is subadditive on \([0,\infty)\) when \(0<r\le1\), while convexity gives \((a+b)^r\le2^{r-1}(a^r+b^r)\) for \(a,b\ge0\) when \(1<r<2\). Hence \(\E\abs{X+Y}^{r}\) and \(\E\abs{X-Y}^{r}\) are finite.

For \(n\in\N^+\), let \(T_n=\{t\in\R:n^{-1}<\abs{t}<n\}\) and define
\[
   \Phi_n(z)=C_r\int_{T_n}\frac{1-\cos(tz)}{\abs{t}^{r+1}}\dd t,
   \qquad z\in\R.
\]
The integrand is non-negative and \(T_n\uparrow\R\setminus\{0\}\). By the Monotone Convergence Theorem applied to the integral in Lemma \ref{lem:chung},
\[
   0\le \Phi_n(z)\le \abs{z}^{r},
   \qquad
   \Phi_n(z)\to \abs{z}^{r}
   \quad\text{for every }z\in\R.
\]
Consequently,
\[
   \abs{\Phi_n(X+Y)-\Phi_n(X-Y)}
   \le \abs{X+Y}^{r}+\abs{X-Y}^{r},
\]
and the right-hand side is integrable. By the Dominated Convergence Theorem,
\begin{equation}\label{eq:dct-limit}
   \E\abs{X+Y}^{r}-\E\abs{X-Y}^{r}
   =\lim_{n\to\infty}\E\{\Phi_n(X+Y)-\Phi_n(X-Y)\}.
\end{equation}
It remains to prove that every term on the right is non-negative.

Fix \(n\in\N^+\). Since
\[
   \abs{\cos(t(X-Y))-\cos(t(X+Y))}\le2
   \quad\text{and}\quad
   \int_{T_n}\frac{\dd t}{\abs{t}^{r+1}}<\infty,
\]
we have
\[
   \int_{T_n}\frac{\E\abs{\cos(t(X-Y))-\cos(t(X+Y))}}{\abs{t}^{r+1}}\dd t
   \le 2\int_{T_n}\frac{\dd t}{\abs{t}^{r+1}}<\infty.
\]
Fubini's theorem applies to the truncated integral, so we can exchange expectation and integration, and
\[
   \E\{\Phi_n(X+Y)-\Phi_n(X-Y)\}
   =C_r\int_{T_n}\frac{\E\{\cos(t(X-Y))-\cos(t(X+Y))\}}{\abs{t}^{r+1}}\dd t.
\]
For fixed \(t\in\R\), the numerator is absolutely integrable, because its absolute value is at most \(2\). Also,
\[
   \cos(t(x-y))-\cos(t(x+y))=2\sin(tx)\sin(ty).
\]
In the discrete model this gives
\[
   \E\{\cos(t(X-Y))-\cos(t(X+Y))\}
   =2\sum_{i\in I}\sum_{j\in I}\sin(ta_i)\sin(ta_j)p_{ij}.
\]
The double series is absolutely convergent because \(\abs{\sin(ta_i)\sin(ta_j)}\le1\) and \(\sum_{i,j}p_{ij}=1\). To see that its value is non-negative when \(I\) is infinite, take finite sets \(F_m\uparrow I\), apply \eqref{eq:discrete-positive} on each \(F_m\) with \(u_i=\sin(ta_i)\), and pass to the limit using absolute convergence. The finite case is the same without the limiting step.

In the density model,
\[
   \E\{\cos(t(X-Y))-\cos(t(X+Y))\}
   =2\iint_{D^2}\sin(tx)\sin(ty)f(x,y)\dd x\dd y.
\]
Here the integrability condition preceding \eqref{eq:density-positive} holds with \(\delta(x)=\sin(tx)\), because \(\abs{\sin(tx)}\le1\) and \eqref{eq:density-normalization} holds. Therefore \eqref{eq:density-positive} implies that the last integral is non-negative.

Thus \(\E\{\Phi_n(X+Y)-\Phi_n(X-Y)\}\ge0\) for every \(n\in\N^+\). Passing to the limit in \eqref{eq:dct-limit} proves the inequality for \(0<r<2\).

Next assume that \(r=2\). Let \((X,Y)\) satisfy either the discrete model or the density model, and assume that \(\E X^2<\infty\). Since \(X\) and \(Y\) have the same marginal distribution, \(\E Y^2<\infty\). Moreover,
\[
   \E\abs{XY}\le \frac12\E X^2+\frac12\E Y^2<\infty.
\]
In the discrete model,
\[
   \E[XY]=\sum_{i\in I}\sum_{j\in I}a_i a_j p_{ij}\ge0.
\]
For \(I\) infinite, this is justified by applying \eqref{eq:discrete-positive} to finite sets \(F_m\uparrow I\) with \(u_i=a_i\), and then passing to the limit by the absolute convergence implied by \(\E\abs{XY}<\infty\). In the density model,
\[
   \E[XY]=\iint_{D^2}xy f(x,y)\dd x\dd y\ge0
\]
by \eqref{eq:density-positive} with \(\delta(x)=x\). Since \(\abs{X+Y}^{2}-\abs{X-Y}^{2}=4XY\), taking expectations gives
\[
   \E\abs{X+Y}^{2}-\E\abs{X-Y}^{2}=4\E[XY]\ge0.
\]
Therefore the inequality also holds for \(r=2\).

It remains to show that the statement fails for every \(r>2\). Fix \(r>2\), and define
\[
   A=2^{\frac{2r}{r-2}},
   \qquad
   p=\frac{r}{2^r A},
   \qquad
   q=1-p.
\]
Since \(2^r>2r\) for \(r>2\) and \(A>1\), we have \(0<p<1/2\), and therefore \(q>1/2\). Let \(X\) and \(Y\) be independent and identically distributed random variables with
\[
   \mathbb P(X=A)=p,
   \qquad
   \mathbb P(X=-1)=q.
\]
This distribution belongs to the discrete model. Ordering the support as \(A,-1\), its joint probability matrix is
\[
   (p_{ij})=\begin{pmatrix}
      p^2 & pq\\
      pq & q^2
   \end{pmatrix},
\]
and it is non-negative in the quadratic-form sense because, for every \((u_1,u_2)\in\R^2\),
\[
   \sum_{i=1}^2\sum_{j=1}^2u_i u_j p_{ij}=(pu_1+qu_2)^2\ge0.
\]
Also \(\E\abs{X}^{r}<\infty\), since the distribution has two points.

Write
\[
   \Delta=\E\abs{X+Y}^{r}-\E\abs{X-Y}^{r}.
\]
The three possible values of \(\abs{X+Y}\) are \(2A\), \(2\), and \(A-1\), with probabilities \(p^2\), \(q^2\), and \(2pq\), respectively. The random variable \(\abs{X-Y}\) is equal to \(A+1\) with probability \(2pq\) and is equal to \(0\) otherwise. Hence
\[
   \Delta
   =p^2(2A)^r+q^2 2^r+2pq(A-1)^r-2pq(A+1)^r.
\]
We now prove directly that \(\Delta<0\). Let \(g(t)=t^{r-1}\) for \(t>0\). Since \(r>2\), the function \(g\) is convex on \((0,\infty)\). Jensen's inequality, applied to the uniform probability measure on \([A-1,A+1]\), gives
\[
   \frac12\int_{A-1}^{A+1}g(t)\dd t
   \ge g\left(\frac12\int_{A-1}^{A+1}t\dd t\right)
   =g(A)=A^{r-1}.
\]
It follows that
\[
   (A+1)^r-(A-1)^r
   =r\int_{A-1}^{A+1}t^{r-1}\dd t
   \ge 2rA^{r-1}.
\]
Using this estimate and \(q>1/2\), we obtain
\begin{align*}
   \Delta
   &=2^r p^2A^r+2^rq^2-2pq\{(A+1)^r-(A-1)^r\} \\
   &\le 2^r p^2A^r+2^rq^2-4rpqA^{r-1} \\
   &=2^rq^2+\frac{r^2}{2^r}A^{r-2}(1-4q) \\
   &<2^r-\frac{r^2}{2^r}A^{r-2} \\
   &=2^r-r^2 2^r \\
   &=2^r(1-r^2) \\
   &<0.
\end{align*}
Therefore \(\E\abs{X+Y}^{r}<\E\abs{X-Y}^{r}\). Thus, when \(r>2\), the inequality does not hold for every random vector in the discrete model.

The same failure also occurs in the density model. Let \(Z_1\) and \(Z_2\) be independent copies of the two-point random variable above. For \(0<\varepsilon<1/2\), let \(U_1\) and \(U_2\) be independent uniform random variables on \([-1,1]\), independent of \(Z_1\) and \(Z_2\), and define
\[
   X_\varepsilon=Z_1+\varepsilon U_1,
   \qquad
   Y_\varepsilon=Z_2+\varepsilon U_2.
\]
Then \(X_\varepsilon\) and \(Y_\varepsilon\) are independent with common density
\[
   h_\varepsilon(x)=\frac{q}{2\varepsilon}\mathbf 1_{[-1-\varepsilon,-1+\varepsilon]}(x)
   +\frac{p}{2\varepsilon}\mathbf 1_{[A-\varepsilon,A+\varepsilon]}(x)
\]
on
\[
   D_\varepsilon=[-1-\varepsilon,-1+\varepsilon]\cup[A-\varepsilon,A+\varepsilon].
\]
Hence \((X_\varepsilon,Y_\varepsilon)\) has joint density
\[
   f_\varepsilon(x,y)=h_\varepsilon(x)h_\varepsilon(y),
   \qquad (x,y)\in D_\varepsilon^2.
\]
This density satisfies \eqref{eq:density-positive}, since
\[
   \iint_{D_\varepsilon^2}\delta(x)\delta(y)f_\varepsilon(x,y)\dd x\dd y
   =\left(\int_{D_\varepsilon}\delta(x)h_\varepsilon(x)\dd x\right)^2\ge0
\]
whenever the left-hand side is absolutely integrable. Also \((X_\varepsilon,Y_\varepsilon)\to(Z_1,Z_2)\) almost surely as \(\varepsilon\downarrow0\). Since the variables are uniformly bounded for sufficiently small \(\varepsilon\), and since
\[
   (x,y)\mapsto \abs{x+y}^{r}-\abs{x-y}^{r}
\]
is continuous, the Dominated Convergence Theorem gives
\[
   \E\abs{X_\varepsilon+Y_\varepsilon}^{r}-\E\abs{X_\varepsilon-Y_\varepsilon}^{r}\to \Delta<0.
\]
Hence the inequality fails for all sufficiently small \(\varepsilon\) in the density model as well. This proves the converse direction, and hence the theorem.
\end{proof}

\begin{corollary}\label{cor:cauchy}
Let \((X,Y)\) have either the joint probability mass function \eqref{eq:cauchy-discrete} or the joint density function \eqref{eq:cauchy-density}. Assume that \(0<r<2\) and \(\E\abs{X}^{r}<\infty\). Then
\[
   \E\abs{X+Y}^{r}\ge \E\abs{X-Y}^{r}.
\]
Consequently, the open question posed by \citet[Remark 3.1]{ZhuangHu2026} is proved.
\end{corollary}

\begin{proof}
First consider the probability mass function \eqref{eq:cauchy-discrete}. Define
\[
   p_{ij}=\frac{d_i d_j}{c_i+c_j},
   \qquad i,j\in I.
\]
The array \((p_{ij})\) is non-negative and symmetric, and it is normalized because \eqref{eq:cauchy-discrete} is assumed to be a probability mass function. To verify \eqref{eq:discrete-positive}, fix a finite set \(F\subseteq I\) and real numbers \((u_i)_{i\in F}\). Put \(\eta_i=u_i d_i\). By Lemma \ref{lem:cauchy-positive}, applied to the finite counting measure on \(F\),
\[
   \sum_{i\in F}\sum_{j\in F}u_i u_j p_{ij}
   =\sum_{i\in F}\sum_{j\in F}\frac{\eta_i\eta_j}{c_i+c_j}\ge0.
\]
Thus \eqref{eq:cauchy-discrete} satisfies the discrete model.

Next consider the density function \eqref{eq:cauchy-density}. It is non-negative, symmetric, and normalized by assumption. Let \(\delta:D\to\R\) be measurable and suppose that
\[
   \iint_{D^2}\abs{\delta(x)\delta(y)}\frac{d(x)d(y)}{c(x)+c(y)}\dd x\dd y<\infty.
\]
Set \(\eta(x)=\delta(x)d(x)\). Then Lemma \ref{lem:cauchy-positive}, applied with Lebesgue measure, gives
\[
   \iint_{D^2}\delta(x)\delta(y)\frac{d(x)d(y)}{c(x)+c(y)}\dd x\dd y
   =\iint_{D^2}\frac{\eta(x)\eta(y)}{c(x)+c(y)}\dd x\dd y\ge0.
\]
Thus \eqref{eq:cauchy-density} satisfies the density model. The result now follows from Theorem \ref{thm:main}.
\end{proof}

\begin{remark}
The case \(r<0\) is outside the scope of absolute moment inequalities. If negative powers are nevertheless considered, with the convention that \(0^r=+\infty\), the analogous inequality fails for every \(r<0\). Let \(X\) and \(Y\) be independent and uniformly distributed on \([1,2]\). This is the density model with \(D=[1,2]\) and \(f(x,y)=1\). Since \(X+Y\in[2,4]\),
\[
   \E\abs{X+Y}^{r}\le 2^r<1.
\]
Also \(0\le \abs{X-Y}\le1\) almost surely. If \(-1<r<0\), then \(\abs{X-Y}^r\ge1\) almost surely and \(\E\abs{X-Y}^r<\infty\), so
\[
   \E\abs{X+Y}^{r}<\E\abs{X-Y}^{r}.
\]
It remains only to explain the case \(r\le -1\). Since \(X\) and \(Y\) are independent and uniform on \([1,2]\), the random variable \(X-Y\) has density
\[
   h(w)=1-\abs{w},
   \qquad -1<w<1.
\]
Therefore \(U=\abs{X-Y}\) has density \(2(1-u)\) on \((0,1)\). Hence
\[
   \E\abs{X-Y}^{r}=2\int_0^1 u^r(1-u)\dd u.
\]
For \(0<u<1/2\), we have \(1-u>1/2\). Consequently,
\[
   2\int_0^1 u^r(1-u)\dd u
   \ge \int_0^{1/2}u^r\dd u=+\infty,
   \qquad r\le -1.
\]
Thus \(\E\abs{X-Y}^{r}=+\infty\), whereas \(\E\abs{X+Y}^{r}<\infty\). The inequality fails for every \(r<0\).
\end{remark}

\end{document}